\documentclass[12pt, a4paper, reqno]{amsart}

\usepackage{amssymb, amsrefs}

\usepackage{hyperref}
\usepackage{setspace}
\numberwithin{equation}{section}

\theoremstyle{plain}
\newtheorem{thm}{Theorem}[section]

\newtheorem{defin}[thm]{Definition}
\newtheorem{rem}[thm]{Remark}

\begin{document}

\title{Exponential Taylor domination}


\begin{abstract}
Let $f(z) = \sum_{k=0}^\infty a_k z^k$ be an analytic function in a disk $D_R$ of radius $R>0$, and assume that $f$ is $p$-valent in $D_R$, i.e. it takes each value $c\in{\mathbb C}$ at most $p$ times in $D_R$. We consider its Borel transform
$$
B(f)(z) = \sum_{k=0}^\infty \frac{a_k}{k!} z^k ,
$$
which is an entire function, and show that, for any $R>1$, the valency of the Borel transform $B(f)$ in $D_R$ is bounded in terms of $p,R$. We give examples, showing that our bounds, provide a reasonable envelope for the expected behavior of the valency of $B(f)$. These examples also suggest some natural questions, whose expected answer will strongly sharper our estimates.

We present a short overview of some basic results on multi-valent functions, in connection with ``Taylor domination'', which, for $f(z) = \sum_{k=0}^\infty a_k z^k$, is a bound of all its Taylor coefficients $a_k$ through the first few of them. Taylor domination is our main technical tool, so we also discuss shortly some recent results in this direction.
%
\end{abstract}


\author{Omer Friedland}
\address{Institut de Math\'ematiques de Jussieu, Sorbonne Universit\'e, 4 place Jussieu, 75252 Paris, France.}
\email{omer.friedland@imj-prg.fr}

\author{Gil Goldman}
\address{Department of Mathematics, The Weizmann Institute of Science, Rehovot 76100, Israel.}
\email{gilgoldm@gmail.com}

\author{Yosef Yomdin}
\address{Department of Mathematics, The Weizmann Institute of Science, Rehovot 76100, Israel.}
\email{yosef.yomdin@weizmann.ac.il}

\maketitle

\section{Introduction}

``Taylor domination'' for an analytic function $f(z) = \sum_{k=0}^\infty a_k z^k$ is an explicit bound of all its Taylor coefficients $a_k$ through the first few of them. This property was classically studied, in particular, in relation with the Bieberbach conjecture, which was finally proved in \cite{dB85}: For a univalent function $f$ we always have $|a_k| \le k|a_1|$ (see \cite{B55,Bi36,C35,H94} and references therein).

To give an accurate definition, let us assume the radius of convergence of the Taylor series for $f$ is $\widehat R$ (for $0<\widehat{R} \le +\infty$).

\begin{defin} [\cite{BY16}]
Let $0<R < \widehat{R}$, $N\in{\mathbb N}$, and let $S(k)$ be a positive sequence of a subexponential growth. A function $f$ has an $(N,R,S(k))$-Taylor domination property if for any $k \ge N+1$ we have
$$
|a_k| R^k \le S(k) \max_{0 \le i \le N} |a_i| R^i .
$$
For
a constant sequence $S(k) \equiv C$, we simply denote this property by $(N,R,C)$-Taylor domination.
\end{defin}

The parameters $N,R$, and $S(k)$ of the Taylor domination are not defined uniquely. In fact, each nonzero analytic function $f$ has this property, with $N$ being the index of its first nonzero Taylor coefficient $a_k$ (see e.g. \cite[Proposition 1.2]{BY16}). Consequently, the property of Taylor domination becomes interesting only for {\it specific classes} of analytic functions, for which we can specify the parameters $N, R, S(k)$ in an explicit and uniform way.

\subsection{Some classical examples}

One of the most important examples is provided by $p$-valent functions: Here Taylor domination with explicit parameters is a difficult result of geometric function theory, which is closely related to the Bieberbach conjecture. Since our main result is in this direction, we provide below some background (see Section \ref{sec:zeroes}).

Another striking example, of roughly the same period (1930-th) is Bautin's discovery (see, \cite{B39,B39-2}, and, e.g. \cite{FraY97, Y14}, and references therein), which is one of the most important sources of uniform Taylor domination: The Taylor coefficients of the function in question are polynomials (analytic functions) in a finite number of the problem's parameters. Taylor domination in this case is (formally) a consequence of Hilbert's finiteness theorem, and of its ``quantitative'' extensions (Hironaka's division algorithm, see e.g. \cite{FraY97} and references therein). Of course, besides discovering a general approach, Bautin provided explicit, and highly non-trivial, calculations for the plane vector fields of degree two, thus obtaining one of the most important results in the second part of Hilbert's 16-th problem up today: At most three limit cycles can bifurcate from a center in a quadratic plane vector field.

One more classical result, which we mention, concerns Taylor domination, with explicit parameters, for rational functions. In a sense, this is a special case of $p$-valent functions, but the known results for rational functions are much sharper. Here Taylor domination is, essentially, equivalent to the important and widely applied ``Turan's lemma'' (see \cite{T53,T84}, and, e.g. \cite{BY16}, and references therein).

\subsection{Some recent developments}

Recently, Taylor domination was also investigated in some additional situations:

\smallskip

\noindent 1. {\it Linear recurrence relations.} Functions whose Taylor coefficients satisfy linear recurrence relations, in particular, of Poincar\'e-Perron type, possess an explicit Taylor domination (see \cite{BY16}).

\smallskip

\noindent 2. {\it $(s,p)$-valent functions.} Functions which preserve some valency bounds after subtracting from them any polynomial of degree $s$. A complete characterization of such functions through linear recurrence relations for Taylor coefficients was obtained in \cite{FriY17}.

\smallskip

\noindent 3. {\it Remez-type inequalities.} These inequalities bound $|f|$ on a disk $D$ through the bound on $|f|$ on a certain subset $\Omega$ of $D$. In \cite{FriY17} a rather accurate Remez-type inequality was obtained for $(s,p)$-valent functions.

\smallskip

\noindent 4. {\it Bautin-type results.} Providing Taylor domination through an explicit description of the Bautin ideals in certain specific cases. Many important results in this direction were obtained in the modern analytic theory of ODE's (\cite{R98} and references therein). Beyond the field of analytic theory of ODE's, some (initial) general results were given in \cite{FraY97,Y97,BY97}, while certain specific problems were treated, via direct calculations, in \cite{Y98,Y14}.

\smallskip

\noindent 5. {\it Efficiently transcendental functions.} In \cite{CY18} we investigate analytic functions $f$ such that a result of a substitution of $f$ as $y$ into a polynomial $P(z,y)$, i.e. $g(z) = P(z,f(z))$ preserves, for any $P$, Taylor domination with explicit parameters, depending only on the degree of $P$. Our main tool is ``linear Bautin ideals'', as in \cite{Y98}. These results are applied in \cite{CY18}, via the Pila-Wilkie approach, to bounding the number of rational points on the graph $y = f(x)$.

\subsection{The scope and the goals of the paper}

As it is clear from the above, the notion of Taylor domination was historically considered mostly for functions $f(z)$ with a finite radius of convergence. The Borel transform (\cite{B1899}) maps such a function $f(z) = \sum_{k=0}^\infty a_k z^k$ into an entire function
$$
g(z) = B(f)(z) = \sum_{k=0}^\infty \frac{a_k}{k!} z^k .
$$

The class of all the images $B(f)(z)$ of functions $f(z)$ having a nonzero radius of convergence, can be easily described explicitly: It consists of all the entire functions $g(z) = \sum_{k=0}^\infty b_k z^k$ with $k!b_k$ growing at most exponentially. There is an integral expression for the inverse of the Borel transform: $f(z) = \int_0^\infty e^{-t} g(tz)dt$, which plays important role in Borel's summation method for divergent series. We expect this formula to be important in our line of research, but we do not use it in the present paper.

Of course, the main applications of the Borel transform are in summation of divergent series. However, also its action on regular (but not extendable to the entire complex plane) functions $f$ was extensively studied. Still, to the best of our knowledge, the problems of the behavior of the ``valency'', and the corresponding problems of Taylor domination, did not get an adequate attention.

The goal of the present paper is to present some initial results in this direction. Given a $p$-valent function $f$, we estimate the valency of $B(f)$ on the disks $D_R$, for given $R>0$ (see Theorem \ref{thm-main} below). We also provide some examples, which outline the degree of the (non)-sharpness of our bounds, and suggest some related questions.


Finally, let us mention some recent observations from \cite{BGY19}, concerning the moment and Fourier reconstruction of the ``spike-train signals'' $F(x) = \sum_{j=1}^d a_j\delta (x-x_j)$. The Fourier transform of $F$ is an exponential polynomial ${\mathbb F}(s) = \sum_{j=1}^d a_j e^{-2\pi i x_js}$, while its Stieltjes transform ${\mathbb S}(z) = \sum_{j=1}^d \frac{a_j}{1-zx_j}$ is a rational function with the poles at $x_j$. It is easy to see that the Taylor coefficient at zero of ${\mathbb F}(s)$ are $\frac{m_k}{k!}$, while the Taylor coefficient at zero of ${\mathbb S}(z)$ are $m_k$. Here
$$
m_k = \int x^k F(x)dx
$$
are the consecutive moments of our signal $F(x)$. In particular, ${\mathbb F}(s)$ is the Borel transform of ${\mathbb S}(z)$. Specifically, we show in \cite{BGY19}, using Taylor domination, that if ${\mathbb F}(s)$ is small on a certain real interval, then all its Taylor coefficients are small. This fact is crucial for comparing accuracy of Fourier and moment reconstructions of spike-train signals. We expect that this result of \cite{BGY19} can be generalized to Borel transforms of general functions, having a Taylor domination property.


\section{Some background on Taylor domination and counting zeroes} \label{sec:zeroes}

By definition, Taylor domination allows us to compare the behavior of $f(z)$ with the behavior of the polynomial $P_N(z) = \sum_{k=0}^N a_k z^k$. In particular, the number of zeroes of $f$, in an appropriate disk, can be easily bounded in this way (see below for more details). However, the opposite direction (bounding zeros implies Taylor domination) is a deep and difficult classical results of geometric function theory, closely related to the Bieberbach conjecture, and going back at least, to Ahlfors. We state here one prominent classical result in this direction \cite{Bi36}. To formulate it accurately, we need the following definition (see \cite{H94} and references therein).

\begin{defin}
Let $f$ be a regular in a domain $\Omega\subset{\mathbb C}$. The function $f$ is called $p$-valent in $\Omega$, if for any $c\in{\mathbb C}$ the number of solutions in $\Omega$ of the equation $f(z) = c$ does not exceed $p$.
\end{defin}

\begin{thm} [Biernacki, 1936, \cite{Bi36}] \label{thm-bier}
If $f$ is $p$-valent in the disk $D_{R}$ of radius $R$ centered at $0\in{\mathbb C}$ then for any $k \ge p+1$
$$
|a_k| R^k \le A(p)k^{2p-1} \max_{1 \le i \le p} |a_i| R^i ,
$$
where $A(p)$ is a constant depending only on $p$.
\end{thm}

In our notations, Theorem \ref{thm-bier} claims that a function $f$, which is $p$-valent in $D_{R}$, has a $(p,R, A(p)k^{2p-1})$-Taylor domination property. For univalent functions, i.e. for $p = 1,R = 1$, Theorem \ref{thm-bier} gives $|a_k| \le A(1)k|a_{1} |$
for any $k$, while the sharp bound of the Bieberbach conjecture is $|a_k| \le k|a_{1} |$.

%

Various forms of inverse results to Theorem \ref{thm-bier} 
are known (the reference list is long, and we skip it here). In particular, an explicit, and reasonably accurate, bound on the number of zeroes of $f$ having a Taylor domination property is given in \cite[Lemma 2.2.3]{RY97}:

\begin{thm} \label{thm:Roy.Yom}
Let $f$ possess an $(N,R,C)$-Taylor domination. Then, for any $R' < \frac{1}{4}R$, the function $f$ has at most $5N+5\log (C+2)$ zeros in $D_{R'}$.
\end{thm}

We can replace the bound on the number of zeroes of $f$ by the bound on its valency, if we exclude $a_0$ in the definition of Taylor domination (or, alternatively, if we consider the derivative $f'$ instead of $f$).

\begin{rem} \label{rem}
It is natural to ask whether functions $f$ having a $(p,R, k^{2p-1})$-Taylor domination property (like $p$-valent functions, according to Biernacki's Theorem \ref{thm-bier}), are, at least, $Kp$-valent, in, say, $D_{\frac{1}{2} R}$. Since Theorem \ref{thm:Roy.Yom} concerns only the case of the constant sequence $S(k) \equiv C$, it is not sharp enough to answer this question. Calculating an optimal $C$ and using Theorem \ref{thm:Roy.Yom}, gives only an order of $p\log p$ bound on the valency of $f$ (compare with the discussion in Section \ref{sec:examples} below).
\end{rem}

\section{Main result}

\begin{thm} \label{thm-main}
Let $f(z) = \sum_{k=0}^\infty a_k z^k$ be a $p$-valent function on the unit disk $D_1$. Then, for any $R>1$, the Borel transform $B(f)(z) = \sum_{k=0}^\infty \frac{a_k}{k!} z^k$ is $q$-valent on $D_{R'}$, where $R' < R$, and
$$
q \lessapprox (1 + \log p+\log R)p +R ,
$$
where $\lessapprox$ means up to universal constants.
\end{thm}

\begin{proof}
Let us start with the following simple remark, which allows us to eliminate the parameter $R$ in a Taylor domination, just by a proper scaling of the independent variable. Indeed, the function $f(z) = \sum_{k=0}^\infty a_k z^k$ has an $(N,R,S(k))$-Taylor domination property if and only if the scaled function
\begin{align} \label{eq-hat}
\widehat f(z) := f(Rz) = \sum_{k=0}^\infty a_k R^k z^k ,
\end{align}
defined on the unit disk, has an $(N,1,S(k))$-Taylor domination property.

By assumption, the function $f$ is $p$-valent on the unit disk $D_1$, and therefore, by Theorem \ref{thm-bier}, for any $k \ge p+1$ we have
$$
|a_k| \le A(p)k^{2p-1} \max_{1 \le i \le p} |a_i| ,
$$
where $A(p)$ is a constant depending only on $p$. Hence, for any $R>1$ and for any $k \ge p+1$, we have
\begin{align} \label{eq-b_k}
|a_k| R^k/k! \le A(p)k^{2p-1} \max_{1 \le i \le p} |a_i| R^k/k! \le A(p) \eta \max_{1 \le i \le p} |a_i| ,
\end{align}
where $\eta = \max_{k \ge p+1} \frac{k^{2p-1} R^k}{k!}$.

Now, since, by our assumptions, $R \ge 1$, the numbers $\frac{R^i}{i!}$ for $1\le i \le p$, grow till $i = [R]$, and then start to decrease, where $[R]$ is the integer part of $R$. Hence, the minimum of these numbers is achieved either with the first of them $i=1$, i.e. $R$, or with the last one $i=p$, which is $\frac{R^p}{p!}$. Therefore, we have
$$
\max_{1 \le i \le p} \frac{|a_i| R^i}{i!} \ge \nu \max_{1 \le i \le p} |a_i| ,
$$
where $\nu = \min \{R, \frac{R^p}{p!} \}$.

Plugging, the above estimate, in \eqref{eq-b_k}, we immediately get that, for any $k \ge p+1$,
\begin{align} \label{eq-TD}
\frac{|a_k| R^k}{k!} \le \frac{A(p) \eta}{\nu} \max_{1 \le i \le p} \frac{|a_i| R^i}{i!} .
\end{align}

By re-scaling the restriction of $B(f)$ to $D_R$ to the unit disk (i.e. the Borel transform of the re-scaled function in \eqref{eq-hat}), we get
$$
B(\widehat f)(z) = \sum_{k=0}^\infty \frac{a_k R^k}{k!} z^k ,
$$
and thus, inequality \eqref{eq-TD} provides a bound on all the Taylor coefficients $\frac{a_k R^k}{k!}$ of $B(\widehat f)$ through its first $p$ ones (excluding the constant term), that is, $B(\widehat f)$ has a $(p,1,A(p)\eta/\nu)\text{-Taylor domination}$ property, which, by the above simple remark, also implies that
$$
B(f) \text{ has a } \big(p,R,(A(p)\eta/\nu)\big)\text{-Taylor domination property} .
$$

Now, we use Theorem \ref{thm:Roy.Yom} for the function $B(f)$ with $N=p$ and $C=A(p)\eta/\nu$, which yields the following bound on the valency $q$ of $B(f)$ in $D_{R'}$, for $R' \le R$:
\begin{align} \label{eq-q}
\nonumber q & \le 5N+5\log (C+2) \\
& \le 5p+ 5\log (A(p) \eta/\nu + 2) .
\end{align}

To complete the proof of Theorem \ref{thm-main}, we need to get an explicit bound on $\eta$. Recall,
$$
\eta = \max_{k \ge p+1} \frac{k^{2p-1} R^k}{k!} .
$$
We shall bound $\eta$ by considering two cases:
$$
\eta = \max \{\eta_1, \eta_2\} ,
$$
where
$$
\eta_1: = \max_{p+1 \le k \le 3p} \frac{k^{2p-1} R^k}{k!} \, , \quad \eta_2: = \max_{k \ge 3p+1} \frac{k^{2p-1} R^k}{k!}.
$$

For $\eta_1$ we use an immediate estimate
$$
\eta_1 \le \frac{(3p)^{2p-1} R^{3p}}{(p+1)!} ,
$$
just taking the maximal possible numerator, and the minimal possible denominator.

In order to estimate $\eta_2$, we proceed as follows: We divide the numerator and the denominator by $k^{2p-1}$, and write $k!$ as
$$
k! = (k-2p-1)! \zeta \, , \quad \text {where} \, \, \zeta = \prod_{j=1}^{2p}(1-\frac{2p-1-j}{k}) \ge \frac{1}{3^{2p}} .
$$

Consequently we get
$$
\eta_2 = \max_{k \ge 3p+1} \frac{k^{2p-1} R^k}{k!} \le 3^{2p} \frac{R^k}{(k-2p-1)!} = 3^{2p} R^{2p-1} \frac{R^{k-2p-1}}{(k-2p-1)!}.
$$

It remains to notice, that the last expression, as a function of $k$, decreases, starting with $k-2p-1 = [R]$. Hence, its maximal value is achieved for $k = [R]+2p-1$. Using Stirling formula, we have $\frac{R^R}{R!} \le e^R$, and hence $\eta_2 \le 3^{2p} R^{2p-1} e^R$. Thus, we conclude
$$
\eta \le \max \left\{\frac{(3p)^{2p-1} R^{3p}}{(p+1)!}, 3^{2p} R^{2p-1} e^R \right\} .
$$

Recall also that $\nu = \min \{R, \frac{R^p}{p!} \}$, and considering \eqref{eq-q}, we conclude
\begin{align*}
q & \le 5p+ 5\log (A(p) \eta/\nu + 2) \\
& \lessapprox p + p\log p + p \log R + R ,
\end{align*}
which completes the proof.
\end{proof}

\section{Some examples} \label{sec:examples}

In this section we give some examples, illustrating Theorem \ref{thm-main}. These examples motivate some natural questions (presumably, open), which we discuss below.

\smallskip

\noindent 1. As the first example, consider the function $f(z) = 1+z+z^2+\dots = \frac{1}{1-z}$. This function is univalent in $D_1$, and its Borel transform is $e^z$. The solutions of the equation $e^z = c$ are all the points $\log c + 2\pi k i$, with whatever brunch of $\log $ chosen. Clearly, in the disk $D_R$ there are at most $\frac{R}{2\pi}$ such points, and for some $c$ this bound is achieved. Therefore, $e^z$ in $D_R$ is $p$-valent, with $p = \frac{R}{2\pi}$. 

\smallskip

\noindent 2. Let us now consider the case of larger $p's$. Let
$$
f_p(z) = (z^p-1)(e^z-1) = -\sum_{k = 1}^p \frac {z^k}{k!} + \sum_{k = p+1}^\infty [\frac{1}{(k-p)!} -\frac{1}{k!}] z^k.
$$
Clearly, $f_p(z)$ is at least $p+\frac{R}{2\pi}$-valent in$D_R$, for any $R$. Indeed, the roots of the two factors in $f_p(z)$ provide the required number of solution of $f_p(z) = 0$. On the other side, $f_p(z)$ is the Borel transform of
$$
\tilde f_p(z) = -\sum_{k = 1}^p z^k + \sum_{k = p+1}^\infty [\frac{k!}{(k-p)!} -1] z^k.
$$
We see that $\tilde f_p(z)$ has a $(p,1,S(k))$-Taylor domination, with $S(k) \sim k^p$. We would expect such functions to be $p$-valent, at least in smaller disks, but with the tools in our possession we can prove only that $\tilde f_p(z)$ is $\sim p\log p$-valent in any disk $D_\rho, \rho <1$ (just estimate the maximum in $k$ of $k^p \rho^k$ and use Theorem \ref{thm:Roy.Yom}). We are not aware of ``counting zeroes'' results, sharp enough to provide $Kp$-valency of a function, having $(p,1,Ck^p)$-Taylor domination. Such a result would be an accurate inversion to the Biernaczki's one (Theorem \ref{thm-bier} above). Thus, a natural question is {\it whether a function, having $(p,1,Ck^p)$-Taylor domination is $K(C)p$-valent in a disk $D_{\frac{1}{2}}$, with the constant $K(C)$ depending only on $C$?} Compare Remark \ref{rem} in Section \ref{sec:zeroes} above.

\smallskip

\noindent 3. Finally, consider $h(z) = e^{z^p} = \sum_{l = 1}^\infty \frac{z^{lp}}{l!}$. The entire function $h(z)$ is the Borel transform of a {\it formal power series}
$$
\widehat h(z) = \sum_{l = 1}^\infty \frac{z^{lp}(lp)!}{l!} ,
$$
which is divergent for any $z\ne 0$. It is easy to see that $h(z)$ is $\sim p\cdot R^p$-valent in the disk $D_R$ for any $R$. Indeed, to solve the equation $h(z) = e^{z^p} = c$, we put $u = z^p$, and first solve $e^u = c$, which gives the solutions $u_k = \log c + 2\pi k i$. Then, from each $u_k$ we get $p$ solutions $\nu_{k,l} = (u_k)^{\frac{1}{p}}$. Notice that these solutions are of absolute value $\sim k^{\frac{1}{p}}$ for $k$ big, and hence $\sim p \cdot R^p$ of them are inside the disk $D_R$. We conclude that $h(z)$ is $\sim p\cdot R^p$-valent. This example suggests the following question: {\it Is it possible to extend the results above to the Borel transforms of divergent series?} This would require extending to some divergent series a notion of Taylor domination, and we expect such an extension to be productive in many other questions in this line.

\end{document}